\documentclass[12pt,reqno]{amsart}
\usepackage{amssymb,latexsym}
\usepackage{graphicx}
\usepackage{fancyhdr}
\usepackage{amsthm}
\usepackage{mathtools}
\usepackage[hyphens]{url}
\usepackage{hyperref}
\usepackage{breakurl}
\usepackage{cite}

\def\Spec{\mathrm{Spec}\textrm{ }}
\def\Gal{\mathrm{Gal}}

\newcommand{\Q}{\mathbb Q}
\newcommand{\R}{\mathbb R}

\newcommand{\op}{\operatorname}

\theoremstyle{plain}
\numberwithin{equation}{section}
\newtheorem{thm}{Theorem}[section]
\newtheorem{theorem}[thm]{Theorem}

\newtheorem{corollary}{Corollary}

\theoremstyle{definition}

\begin{document}

\setcounter{page}{1}

\title[Class numbers of large degree nonabelian number fields]{Class numbers of large degree nonabelian number fields}

\author{Kwang-Seob Kim}
\address{School of Mathematics\\
                Korea Institute for Advanced Study(KIAS)\\
                Seoul, 130-722, Korea}
\email{kwang12@kias.re.kr}
\author{John C. Miller}
\address{Department of Applied Mathematics \& Statistics\\
                Johns Hopkins University\\
                100 Whitehead Hall\\
                3400 North Charles Street\\
                Baltimore, MD 21218}
\email{jmill268@jhu.edu}
%\thanks{This material is based upon work supported by the National Science Foundation under Grant No. DUE-1022574.}


\begin{abstract}
If a number field has a large degree and discriminant, the computation of the class number becomes quite difficult, especially without the assumption of GRH.  In this article, we will unconditionally show that a certain nonabelian number field of degree 120 has class number one.  This field is the unique $A_5 \times C_2$ extension of the rationals that is ramified only at 653 with ramification index 2.  It is the largest degree number field unconditionally proven to have class number 1.

The proof uses the algorithm of Gu\`ardia, Montes and Nart to calculate an integral basis and then finds integral elements of small prime power norm to establish an upper bound for the class number; further algebraic arguments prove the class number is 1.  It is possible to apply these techniques to other nonabelian number fields as well.
\end{abstract}

\maketitle

\section{Introduction}
One of the problems of algebraic number theory is to get deep knowledge of the Galois groups of various Galois extensions of number fields, especially maximal extensions of number fields with restricted ramification. Such Galois groups can be regarded as \'etale fundamental groups of spectra of algebraic integer rings punctured at some closed points, and they play essential roles for understanding the arithmetic of number fields, analogous to the role geometric fundamental groups of manifolds do in geometry. In other words, the \'etale fundamental group $G= \pi^{\textrm{\'et}}_1(\Spec \mathcal{O}_K)$ of the ring of integers $\mathcal{O}_K$ of a number field $K$ is isomorphic to $\Gal(K^{f}_{ur}/K)$ where $K^f_{ur}$ is the maximal extension of $K$ which is unramified over all finite places. This is one of the motivations for the study of unramified extensions of number fields and their Galois groups. In general, one can know the abelianizations of the  \'etale fundamental group $G$ by examining the ideal class group of $\mathcal{O}_K$.  Yamamura's results \cite{Yamamura-1997} tell us that $K^f_{ur}=K_l$, where $K$'s are imaginary quadratic fields with absolute discriminant $|d_K| \leq 420$ and $K_l$ is the top of the class field tower of $K$. Hence we can find examples of abelian or solvable \'etale fundamental groups.

However, one can get little information on the structure of $G$ itself via class field theory, especially in the case where $G$ is a nonabelian simple group.

The first author has previously \cite{Kim-2014} given an example of $K$ with $\Gal(K^{f}_{ur}/K) \simeq A_5$ under the assumption of the generalized Riemann hypothesis (GRH). Let $K$ be a quadratic number field $\Q(\sqrt{653})$ and $L$ be a splitting field of
\begin{equation}
\begin{split} \label{poly1}
x^5 +3x^3 +6x^2 + 2x + 1,
\end{split}
\end{equation}
a polynomial that has complex roots. The field $L$ is an $A_5$-extension of $\Q$ and $653$ is the only prime ramified in this field, with ramification index two \cite[p. 21]{Cais-2005}.  Moreover, $L$ is the unique such field \cite{Basmaji-1994}.

By Abhyankar's lemma, the compositum $KL$ is an extension of $K$ which is unramified over all finite places, and $KL$ is the unique $A_5 \times C_2$ extension of $\Q$ ramified only at 653 with ramification index 2.  Under the assumption of GRH, the first author proved \cite{Kim-2014} that the class number of $KL$ is 1 and futhermore that $K^f_{ur}= KL$.

We have the following natural question: Is it possible to show that $K^f_{ur}=KL$ without assuming GRH, thereby proving unconditionally the existence of a non-solvable \'etale fundamental group?

To do this, the first step is to unconditionally prove that $KL$ has class number one. If a number field has a large degree and discriminant, the computation of the class number becomes quite difficult.  The Minkowski bound is too large to be useful, and the root discriminant of $KL$ is too large to be treated by Odlyzko's unconditional discriminant bounds.  However, the second author showed \cite{MillerWeber} that by finding nontrivial lower bounds for sums over prime ideals of the Hilbert class field, upper bounds can be established for class numbers of fields of larger discriminant.

In this paper, we unconditionally prove that $KL$ has class number 1.  In fact, $KL$ has degree 120, and it is the largest degree number field proven unconditionally to have class number 1.  Previously, among such fields, the one with largest degree, the real cyclotomic field of conductor 151, has degree 75 \cite{MillerPrime}. 

\begin{theorem}\label{MainResult}
The unique $A_5 \times C_2$ extension of $\Q$ ramified only at 653 with ramification index 2 has class number 1.
\end{theorem}

\section{Remark on the maximal unramified extension of $\mathbb{Q}(\sqrt{p})$}
The first author proved \cite{Kim-2014} that if the class number of $KL$ is less than 16, then the class number of $KL$ is 1.  Using discriminant lower bounds, it was further proved, \emph{under the assumption of GRH}, that the class number of $KL$ is indeed less than 16 (and therefore is 1), and also that the maximal unramified extension of $K$ is $KL$.

\begin{thm}(Main theorem of \cite{Kim-2014})\label{KimMainTheorem}
Let $K$ be a real quadratic field $\Q(\sqrt{p})$ with narrow class number $1$, where $p$ is a prime congruent $1$ mod $4$ and $p \neq 5$. Suppose that there exists a totally imaginary $A_5$-extension $L$ over $\Q$ and $p$ is the only prime ramified in this field with ramification index $2$. If $\sqrt{p}<B(1920,0,960)$, then the class number of $KL$ is one and $K^f_{ur}=KL.$

($B(n, r_1, r_2)$ is defined as an infimum of $|d_F|^{1/n_F}$ over all number fields $F$ satisfying $n_F \geq n$ and $\frac{r_1(F)}{n_F}=\frac{r_1}{n}$ (resp. $\frac{r_2(F)}{n_F}=\frac{r_2}{n}$) where $r_1(F)$ (resp. $r_2(F)$) is the number of real (resp. complex) places of a number field $F$.)
\end{thm}
In $B(1920,0,960)$, the number $1920$ represents
\begin{displaymath}
16 \times \textrm{ the degree of }KL.
\end{displaymath}

In particular, the proof of Theorem \ref{KimMainTheorem} has the following consequence:

\begin{corollary}\label{Cor}
With the notations and conditions being the same as above, if the class number of $KL$ is smaller than $16$, then the class number of $KL$ is exactly one.
\end{corollary}

Here we briefly sketch the main elements of the proof, the details of which are given in \cite{Kim-2014}.

\begin{proof}
Let $M$ denote the Hilbert class field of $KL$, and let $M^{(q)}$ denote the Hilbert $q$-class field for a prime $q$.  Every automorphism group of an abelian group of order less than 16 does not contain a subgroup isomorphic to $A_5$, so $\op{Gal}(M^{(q)}/K)$ is a central extension of $A_5$ by $\op{Gal}(M^{(q)}/KL)$.  If $q$ is odd, $\op{Gal}(M^{(q)}/K)$ has a nontrivial abelian quotient \cite[Cor 2.2]{Kim-2014}, contradicting that the narrow class number of $K$ is 1.  Thus the $q$-class group of $KL$ is trivial for odd $q$.

Similarly, $\op{Gal}(M^{(2)}/K)$ has a nontrivial abelian quotient if $\op{Gal}(M^{(2)}/KL)$ has order greater than 2 \cite[Cor 2.3]{Kim-2014}.  Also, if $\op{Gal}(M^{(2)}/KL)$ is trivial, then we are done.  Otherwise, we have $\op{Gal}(M^{(2)}/K)$ isomorphic to $A_5 \times C_2$ or $\op{SL}_2(\mathbb{F}_5)$ \cite[Cor 2.3]{Kim-2014}.  However, $\op{Gal}(M^{(2)}/K)$ can not be isomorphic to $A_5 \times C_2$ because $K$ has narrow class number 1.  Finally, it has been shown \cite[p. 116--118]{Kim-2014} that $\op{Gal}(M^{(2)}/K) \simeq \op{SL}_2(\mathbb{F}_5)$ also leads to a contradiction.
\end{proof}

To prove unconditionally that $KL$ has class number 1, it remains to be shown that the class number is less than 16.

\section{Upper bounds on class numbers of totally complex fields}
The root discriminant $|d_{KL}|^{1/120}$ of $KL$ is approximately 25.5539.  If we assume the generalized Riemann hypothesis, we can use discriminant lower bounds \cite{OdlyzkoTable, Martinet-1980} to show that any totally complex field with degree 480 or larger must have root discriminant larger than 26.48.  But the root discriminant of $KL$ is equal to the root discriminant of its Hilbert class field, so under GRH the class number of $KL$ must be less than $4$.

However, without the assumption of GRH, this method fails for totally complex fields with root discriminant above $4\pi e^\gamma \approx  22.3816$.  To make further progress, we must find another approach that can handle large root discriminants.  Such a method was introduced in \cite{MillerWeber} for totally real fields.  By finding sufficiently many integral elements with small prime norm, an upper bound for class numbers could be established even for fields of large discriminant.

We prove a similar result of totally complex fields.

\begin{theorem}\label{UpperBound}
Let $K$ be a totally complex Galois number field of degree $n$, and let
$$F(x) = \frac{e^{-\left(x/c\right)^2}}{\cosh{\frac{x}{2}}}$$
for some positive constant $c$.  Suppose $S$ is a subset of the prime integers which are unramified in $K$ and factor into principal prime ideals of $K$ of degree $f_p$.  Let
$$B = \gamma + \log{8\pi}  - \log \operatorname{rd}(K) - \int_0^\infty \frac{1-F(x)}{2\sinh{\frac{x}{2}}} \, dx + 2 \sum_{p \in S} \sum_{m=1}^\infty \frac{\log p}{p^{f_p m/2}}F(f_p m \log p),$$
where $\gamma$ is Euler's constant.  If $B > 0$ then we have an upper bound for the class number $h$ of $K$,
$$h < \frac{2c \sqrt{\pi}}{nB}.$$
\end{theorem}

\begin{proof}
The proof is a modification of the argument in \cite{MillerWeber}. We apply Poitou's version \cite{Poitou} of Weil's ``explicit formula" for the Dedekind zeta function of Hilbert class field $H(K)$ of $K$:
\begin{align*}
\log d(H(K)) &= hr_1\frac{\pi}{2} + hn(\gamma + \log 8\pi) - hn \int_0^{\infty} \frac{1-F(x)}{2 \sinh \frac{x}{2}} \, dx - hr_1 \int_0^{\infty} \frac{1-F(x)}{2 \cosh \frac{x}{2}} \, dx \\ &- 4 \int_0^{\infty} F(x) \cosh \frac{x}{2} \, dx + \sum_{\rho} \Phi(\rho) + 2 \sum_{\mathfrak P} \sum_{m=1}^{\infty} \frac{\log N\mathfrak P}{N\mathfrak P^{m/2}}F(m \log N\mathfrak P)
\end{align*}
where $\gamma$ is Euler's constant and $r_1 = 0$ since $K$ is totally complex.  The first sum is over the nontrivial zeros of the Dedekind zeta function of $H(K)$, the second sum is over the prime ideals of $H(K)$, and $\Phi$ is defined by
$$\Phi(s) = \int_{-\infty}^{\infty} F(x) e^{(s - 1/2)x} \, dx.$$

By our choice of $F$, the real part of $\Phi(s)$ is nonnegative everywhere in the critical strip.  Indeed, on the boundary of the critical strip, the real part
$$\op{Re} \Phi(s) = \op{Re} \int_{-\infty}^\infty  \frac{e^{-\left(x/c\right)^2}}{\cosh{\frac{x}{2}}} e^{(s - 1/2)x} \, dx = \int_{-\infty}^\infty e^{-\left(x/c\right)^2} \cos(x \op{Im} s) \, dx = c \sqrt{\pi} e^{-(c \op{Im} s / 2)^2} > 0$$
is positive,  and $\op{Re} \Phi(s) \rightarrow 0$ as $| \op{Im} s| \rightarrow \infty$, so by the maximum modulus principle for harmonic functions, $\op{Re} \Phi(s)$ can not be negative anywhere in the critical strip.

Since the root discriminant $\op{rd}(K)$ of $K$ equals the root discriminant of $H(K)$, we have
$$\log d(H(K)) = hn \log \op{rd}(H(K)) = hn \log \op{rd}(K),$$
and also
$$ 4\int_0^\infty F(x) \cosh{\frac{x}{2}} \, dx = 2c\sqrt{\pi}.$$
We therefore we get the expression
\begin{align*}
hn \log \op{rd}(K) &= hn(\gamma + \log 8\pi) - hn \int_0^\infty \frac{1-F(x)}{2\sinh{\frac{x}{2}}} \, dx \\ &- 2c\sqrt{\pi} + \sum_{\rho} \Phi(\rho) + 2 \sum_{\mathfrak P} \sum_{m=1}^{\infty} \frac{\log N\mathfrak P}{N\mathfrak P^{m/2}}F(m \log N\mathfrak P).
\end{align*}
We rearrange this to get the identity
\begin{equation} \label{eqref:ClassNumberIdentity}
h = \frac{2c\sqrt{\pi}}{n\left[\gamma + \log{8\pi} - \mathcal{G}(F) - \log \operatorname{rd}(K) + \frac{1}{hn} \sum_{\rho} \Phi(\rho) + \frac{2}{hn} \sum_{\mathfrak P} \sum_{m=1}^{\infty} \frac{\log N\mathfrak P}{N\mathfrak P^{m/2}}F(m \log N\mathfrak P)\right]}
\end{equation}
where
$$\mathcal{G}(F) = \int_0^\infty \frac{1-F(x)}{2\sinh{\frac{x}{2}}} \, dx.$$

To get an upper bound for the class number $h$, we need to bound from below the sum over the zeros and the sum over the primes.  The sum $\sum_\rho \Phi(\rho)$ over the critical zeros is nonnegative since the real part of $\Phi(s)$ is nonnegative on the critical strip.  We note that principal ideals in $K$ totally split in the Hilbert class field of $K$.  To find a nontrivial lower bound for the sum over prime ideals of the Hilbert class field, we consider the contribution of the $hn/f_p$ prime ideals $\mathfrak{P}$ of degree $f_p$ that lie over some unramified rational prime $p$:
$$\frac{2}{hn} \sum_{\mathfrak P | p} \sum_{m=1}^{\infty} \frac{\log N\mathfrak P}{N\mathfrak P^{m/2}}F(m \log N\mathfrak P) = 2 \sum_{m=1}^{\infty} \frac{\log p}{p^{f_p m/2}}F(f_p m \log p).$$
Summing this contribution over an arbitrary set $S$ of unramified primes gives a lower bound for the sum over the prime ideals, proving the theorem.
\end{proof}

\section{An integral basis for  $KL$}
In order to apply Theorem \ref{UpperBound}, we must find sufficiently many integral elements of small prime power norm.  To do this, we first must compute a basis of the ring of integers $\mathcal{O}_{KL}$ of $KL$.

\subsection{Computing an integral basis for $KL$}
In general, it is difficult to compute an integral basis for a number field with such large degree and takes an unfeasibly long time using the commonly implemented algorithms. Fortunately, Jordi Gu\`ardia, Jes\'us Montes and Enric Nart studied and recently implemented an algorithm that allows for fast computation of an integral basis. Their algorithm has excellent heuristic running times and low memory requirements. The detailed algorithms are described in \cite{GMN08a}, \cite{GMN08} and \cite{GMN09}.  Even though their algorithms depend on a conjecture, their program verifies that the returned basis indeed generates a maximal order and warns the user if the test fails. The ``Montes package" is available at their homepage \cite{GMN}.

\subsection{Application of Montes package}
Given a number field $F=\Q(\theta)$ defined by an irreducible monic polynomial $f(x)$ with integer coefficients, and a set $S$ of prime divisors of the discriminant of $f$, the Montes package computes an integral basis of the ring of integers $\mathcal{O}_F$.

We must find a polynomial $f(x)$ defining our field $KL$.  Magma \cite{Magma} can use the polynomial (\ref{poly1}) to find the defining polynomial of $KL$, but the coefficients and the discriminant of such polynomial are extremely high.

Instead, to use the Montes package efficiently, we must find a polynomial $f(x)$ with relatively small coefficients.  From the work of Basmaji and Kiming \cite{Basmaji-1994}, it is known that there is a unique $A_5$-extension of $\Q$ which is ramified only at $653$ with ramification index $2$, and so there is a unique $A_5 \times C_2$-extension of $\Q$ ramified only at $653$ with ramification index $2$.  We can also check that the splitting field of
\begin{equation}
\begin{split} \label{poly2}
x^{12} - 2x^{10}+ x^8 - 3x^6 + 2x^4 + 4x^2 + 1
\end{split}
\end{equation}
(see \cite{Kluners}) is an $A_5 \times C_2$-extension of $\Q$, and ramified only at $653$ with ramification two. Therefore the splitting field of (\ref{poly2}) is isomorphic to $KL$, and we find that the defining polynomial $f(x)$ of the splitting field has reasonably-sized coefficients:\\

\scriptsize
\noindent$f(x) = x^{120}  - 176x^{118} + 15344x^{116} - 885072x^{114} + 38020296x^{112} - 1296203136x^{110} +
36449636456x^{108} -  866693813144x^{106} + 17715301473188x^{104} -  314682684335216x^{102} +
 4890917890431160x^{100} - 66745334367421208x^{98}  + 
800278065681872252x^{96} - 
8414962838523553288x^{94} + 77349309250789184324x^{92} -
621279221857410888820x^{90} + 4437488571797416987014x^{88} -
30447148395598893119200x^{86} + 239364400000982953784360x^{84} - \\
2416952862456294977028824x^{82} + 
27290577696141421585067364x^{80} -
289876697936361974353970312x^{78} + \\
 2693760290611273584666952580x^{76} -
 21533137718506216924328804876x^{74} + 148669653143323458431804717338x^{72} -
 907545304473482219783345392936x^{70} + 5232054111925338986843052743996x^{68} - \\
 3239752789658134 6435817828886076x^{66} +
238759214300169626783192907327914x^{64} -\\
 1969821359271537298041753933477564x^{62} + 
 15955962191002947042308447201390658x^{60} - \\
 117632708230681535563853194507331876x^{58} +
 770009787141724893768537360158911141x^{56} - \\
 4460494378285562956504366340932390160x^{54} + 
 22960862707990118711101046634558650520x^{52} -\\
 105831011418214266323798536925298786752x^{50} + 
 441293370897690876129417922335110659940x^{48} -\\
 1684660626364623687429016680021161333136x^{46} + 
 5956123536650235699963217176615899314260x^{44} -\\
 19650574183334889494369045204655644720528x^{42} + 
 60206468968076652065068820738580687974322x^{40} - \\
 166917939895991551434257451990602974231984x^{38} +
 404315032449068740770618799036599277180644x^{36} -\\
 847390146753292599258653510942927227817392x^{34} +
 1580516372974532552131359517512809032793374x^{32} -\\
 2604518872401248539291715628048115698581488x^{30} +
 3304472359602082247414288730787304921769922x^{28} - \\
 2113035564468006137097948259507305236132140x^{26} - 
 2457065881132741431408225647650890813936411x^{24} + 
 2870618480975712690163264954685448459044352x^{22} + 
 47763253056838839358739564013833428094819356x^{20} - 
 177683014438988024196991712904386612981591688x^{18} + 
 259011825112325478055004327935231086308339338x^{16} - 
 586374161577729185199715921998165891986350488x^{14} + 
 1361315854800880637809679080665560920390820298x^{12} - 
 333978908186182311544310995212617151844172044x^{10} - 
 313006045225493493099794359983814771145110707x^8 - 
 2020311228903302658393839253600110446480900072x^6 - 
 1766970645733800499228543786156151879899061482x^4 + 
 2783835141502321145559680957914563686477616612x^2 + 
 6348866967722939861415527094976230039943905289$
\normalsize\\

Now let $a$ be a zero of the above polynomial, i.e.\ $KL=\Q(a)$. We can get an integral basis 
\begin{displaymath}
\mathcal{B}  = (b_1, b_2, \dots, b_{120})
\end{displaymath}
for $KL$ by using the Montes package and Magma software, where each $b_i \in \Q[a]$.  The basis is included as an ancillary file with the arXiv submission of this paper.

\subsection{Finding a better integral basis for $KL$}
Let $\sigma_1, \bar{\sigma}_1, \sigma_2, \bar{\sigma}_2, \dots, \sigma_{60}, \bar{\sigma}_{60}$ be the 60 conjugate pairs of embeddings of $LK$ into the complex numbers.  We have the usual embedding $\iota$ of $KL$ into $\mathbb{R}^{120}$,
\begin{align*}
\iota: KL &\lhook\joinrel\relbar\joinrel\rightarrow \mathbb{R}^{120}\\
x &\longmapsto (\op{Re} \sigma_1(x), \op{Im} \sigma_1(x), \dots, \op{Re} \sigma_{60}(x), \op{Im} \sigma_{60}(x))
\end{align*}
Under this map, ring of integers $\mathcal{O}_{KL}$ is embedded as a lattice in $\mathbb{R}^{120}$.

With some abuse of notation, we respectively define the \emph{multiplicative norms} $N(x)$ of $x \in KL$ and $N(y)$ of $y \in \mathbb{R}^{120}$ to be
$$N(x) = | N_{KL/\mathbb{Q}}(x)|,$$
$$N(y) = \prod_{i=1}^{60} (y_{2i-1}^2 + y_{2i}^2),$$
so that $N(x) = N(\iota(x))$.

In order to find an upper bound for the class number of $KL$, we must find sufficiently many integral elements with small multiplicative norm.  For this purpose, we desire a ``nice" integral basis for $\mathcal{O}_{KL}$ in the sense that embedded basis elements have small Euclidean lengths in $\R^{120}$.  For the basis $\mathcal{B}$ found by the Montes package in the previous subsection, the Euclidean lengths range from $|\iota(b_1)| = \sqrt{60}$ to a rather large  $|\iota(b_{120})| \approx 1.59 \times 10^{70}$, so we do not yet have a ``nice" basis.

We thereby encounter the classical problem of lattice basis reduction.  We can apply the Lenstra-Lenstra-Lov\'asz (LLL) algorithm \cite{LLL}, but it only gets us so far:  A single application of the LLL algorithm in Maple to our initial basis $\iota(\mathcal{B})$ yields a new basis in $\mathbb{R}^{120}$ with Euclidean norms ranging from $\sqrt{60}$  to approximately $64.51$, a substantial improvement yet not good enough for our goal of finding integral elements of small multiplicative norm.  Indeed, a vector in $\mathbb{R}^{120}$ with Euclidean norm of 64.51 can have a multiplicative norm larger than $10^{110}$.

There exist several other methods of lattice basis reduction that can produce shorter vectors than the LLL algorithm at the expense of longer running time, but here we take a somewhat naive approach that is easy to implement in Maple, using its native implementation of the LLL algorithm.

Suppose that, given a list of basis vectors $v_1, v_2, \dots, v_n$ in $\mathbb{R}^n$, we want apply the LLL algorithm to produce a basis of relatively short vectors.  Typically, for lattices of high dimension and vectors of widely varying Euclidean length, the results can be improved if the $v_i$ are first ordered in ascending length prior to applying the LLL algorithm.  This motivates the following approach:
\begin{enumerate}
\item [Step 1:]	Sort the vectors so that the $v_i$ are in ascending length.
\item [Step 2:]	Apply the LLL algorithm to the sorted list of vectors to produce a new list of basis vectors $v_i$.
\item [Step 3:]	Repeat Steps 1 and 2 until the desired shortness of basis vectors is achieved, or until repeating Steps 1 and 2 no longer produces bases with shorter vectors.
\end{enumerate}

Perhaps surprisingly, this simple method works quite well with the integral basis $\mathcal{B}$ of $\mathcal{O}_{KL}$.  As mentioned above, we first applied LLL to $\iota(\mathcal{B})$ to get a new basis.  Then we  applied Steps 1 and 2 repeatedly (in fact, 279 times) until we obtained another basis
$$\mathcal{C} = (c_1, c_2, \dots, c_{120}), \quad c_i \in \mathbb{R}^{120}$$
that has sufficiently short $c_i$.  The $c_i$ have Euclidean norms that range from $|c_1| = \sqrt{60}$ to $|c_{120}| \approx 15.30$.  The inverse of change of basis matrix to get from $\iota(\mathcal{B})$ to $\mathcal{C}$ can be found as an ancillary file with the arXiv submission of this paper.

\section{Finding elements of $\mathcal{O}_{KL}$ with small multiplicative norm}
Our next step is to find integral elements of our number field $KL$ with small multiplicative norm.  Then we can apply the theorem from \cite{MillerWeber} to establish an upper bound for the class number.  To find such elements of small norm, we will search over ``sparse" vectors, i.e.\ vectors where almost all the coefficients are zero with respect to the basis $(c_j)$.

Table \ref{TablePrimeNorm} lists the elements of small prime norm that are found by searching over ``sparse" vectors.  These prime integers generate principal ideals which totally split in $KL$ into 120 principal prime ideals, each of which is generated by the given element or one of its Galois conjugates.  Therefore we can include the primes given in Table \ref{TablePrimeNorm} as degree 1 primes in the set $S$ used by Theorem \ref{UpperBound}.

\begin{center}
\begin{table}
\caption{Generators of some small degree 1 primes in $\mathcal{O}_{KL}$}
\label{TablePrimeNorm}
\begin{tabular}{| c | c | c | }
\hline
  Element & Norm \\ \hline
  $c_2 + c_{17} + c_{46}$ 		& $3571$ \\
  $c_3 - c_9 + c_{48}$           	& $5477$ \\
  $c_{14} - c_{41}$ 			& $7499$ \\
  $c_1 - c_{48} + c_{64}$		& $8867$ \\
  $c_{11} + c_{75}$ 		  	& $15679$ \\
  $c_2 - c_{54} + c_{70}$     	& $17203$ \\
  $c_3 + c_{12} - c_{100}$        	& $20047$ \\
  $c_{23} + c_{104}$		       	& $25343$ \\
  $c_{25} - c_{74} $   	   		& $31477$ \\
  $c_{10} - c_{21} + c_{97} $   	& $34613$ \\
  $c_{49} - c_{80} $		         	& $35537$ \\
\hline
\end{tabular}
\quad
\begin{tabular}{| c | c | c | }
\hline
  Element & Norm \\ \hline
  $c_{71} + c_{74} $ 		         	& $43787$ \\
  $c_{30} + c_{62} $    		& $44879$ \\
  $c_1 - c_{17} + c_{77} $   	& $45361$ \\
  $c_{95} $			   	& $46271$ \\
  $c_2 - c_{28} - c_{62} $   	& $48341$ \\
  $c_{23} - c_{53} + c_{85} $   	& $54311$ \\
  $c_2 + c_{31} + c_{76} $   	& $95327$ \\
  $c_{36} + c_{49} $     		& $111611$ \\
  $c_3 + c_{22} + c_{66} $   	& $113081$ \\
  $c_{23} - c_{62} $		         	& $137927$ \\
  $c_7 - c_{89} $		   	& $139999$ \\
\hline
\end{tabular}
\end{table}
\end{center}

Table \ref{TableHigherDegree} lists some elements which have norms of small prime power.  However, this is not useful to us unless we can show that the ideals generated by the rational primes factor in $KL$ into degree 2 or 3 primes, rather than totally split.  Since $KL$ is the splitting field of
$$x^{12} - 2x^{10}+ x^8 - 3x^6 + 2x^4 + 4x^2 + 1,$$
it suffices and is straightforward to check that this polynomial does not totally split mod $p$ for each of the primes in Table \ref{TableHigherDegree}.  Therefore each of these primes can be included in set $S$ as degree 2 primes, except for 11 which is included as degree 3.

\begin{center}
\begin{table}
\caption{Generators of some small degree 2 and degree 3 primes in $\mathcal{O}_{KL}$}
\label{TableHigherDegree}
\begin{tabular}{| c | c | c | }
\hline
  Element & Norm \\ \hline
  $c_9 + c_{62}$	 		& $13^2$ \\
  $c_{71}$		 			& $83^2$ \\
  $c_{94}$					& $89^2$ \\
  $c_4 + c_{17} - c_{76} $	        	& $137^2$ \\
  $c_8 - c_{62} $ 		         	& $227^2$ \\
\hline
\end{tabular}
\quad
\begin{tabular}{| c | c | c | }
\hline
  Element & Norm \\ \hline
  $c_{69}$			    		& $229^2$ \\
  $c_{22} + c_{25} $    		& $251^2$ \\
  $c_{24} - c_{75} $ 		         	& $383^2$ \\
  $c_7 + c_{42} $ 		         	& $433^2$ \\
  $c_1 - c_7 + c_{68} $    		& $11^3$ \\
\hline
\end{tabular}
\end{table}
\end{center}

However, we have still not found enough principal prime ideals to establish a sufficiently strong upper bound for the class number of $KL$.  To find more principal prime ideals, we consider elements that generate principal ideals that are products of two prime ideals, one of which is already known to be principal.  Essentially, we are just finding relations in the class group.  Table \ref{TableComposite} list some integral elements and their norms.  Consider, for example, the element $c_2 + c_{65}$, which has norm $13^2 \cdot 19^2$.  Since we also know that $c_9 + c_{62}$ generates a degree 2 prime ideal of norm $13^2$, we can divide $c_2 + c_{65}$ by the appropriate Galois conjugate of $c_9 + c_{62}$ to find an integral element of norm $19^2$.  Since $19$ does not totally split in $KL$, we know that $19$ factors into degree 2 principal prime ideals.  Similarly, we can show that 109 factors into degree 2 principal prime ideals, and the primes 7 and 23 factor into degree 3 principal prime ideals.  Also we can show that the following primes totally split in $KL$ into degree 1 principal prime ideals:
$$6361, 10753, 11681, 12619, 16561, 19963, 23431,$$
$$23531, 32309, 33403, 41621, 48179, 56359, 58601.$$

\begin{center}
\begin{table}
\caption{Generators of some composite ideals in $\mathcal{O}_{KL}$}
\label{TableComposite}
\begin{tabular}{| c | c | c | }
\hline
  Element & Norm \\ \hline
  $c_2 + c_{65}$	 		& $13^2 * 19^2$ \\
  $c_9 - c_{89}$	 		& $13^2 * 7^3$ \\
  $c_{24} - c_{70}$	 		& $13^2 * 6361$ \\
  $c_{35} + c_{91}$	 		& $13^2 * 10753$ \\
  $c_{68} + c_{93}$	 		& $13^2 * 11681$ \\
  $c_{18} + c_{91}$	 		& $13^2 * 109^2$ \\
  $c_3 + c_{39} - c_{74}$	 	& $13^2 * 23^3$ \\
  $c_1 + c_{68} + c_{78}$	 	& $19^2 * 12619$ \\
  $c_2 - c_{21} + c_{44}$	 	& $13^2 * 16561$ \\
\hline
\end{tabular}
\quad
\begin{tabular}{| c | c | c | }
\hline
  Element & Norm \\ \hline
  $c_{24} - c_{80}$		 	& $13^2 * 19963$ \\
  $c_1 + c_{65} + c_{77}$		& $19^2 * 23431$ \\
  $c_{53} - c_{89}$			& $19^2 * 23531$ \\
  $c4 + c_{25} - c_{61}$		& $13^2 * 32309$ \\
  $c_{80}$					& $31147 * 33403$ \\
  $c_{15} + c_{100}$			& $13^2 * 41621$ \\
  $c_1 - c_{24} + c_{72}$		& $13^2 * 48179$ \\
  $c_{61} + c_{79}$			& $13^2 * 56359$ \\
  $c_{56} - c_{71}$			& $13^2 * 58601$ \\
\hline
\end{tabular}
\end{table}
\end{center}

\section{An upper bound for the class number of $KL$}
We are now in a position to prove our main result by applying Theorem \ref{UpperBound} and Corollary \ref{Cor}.

\begin{proof}[Proof of Theorem \ref{MainResult}]
By searching for elements of small norm, and taking quotients where necessary, we find a number of primes that can be included in the set $S$ of unramified primes that factor into principal prime ideals.  In particular, the following 36 primes totally split into degree 1 principal primes in $KL$:
$$3571, 5477, 6361, 7499, 8867, 10753, 11681, 12619, 15679, 16561, 17203, 19963, 20047,\\$$
$$23431, 23531, 25343, 31477, 32309, 33403, 34613, 35537, 41621, 43787, 44879, 45361,\\$$
$$46271, 48179, 48341, 54311, 56359, 58601, 95327, 111611, 113081, 137927, 139999.$$
Also, the following 11 primes factor into degree 2 principal primes in $KL$:
$$13, 19, 83, 89, 109, 137, 227, 229, 251, 383, 433.$$
Finally, the following three primes factor into degree 3 principal primes in $KL$:
$$7,11,23.$$
If we include these primes in our set $S$ and set $c=24.5$, we find that
$$2 \sum_{p \in S} \sum_{m=1}^\infty \frac{\log p}{p^{f_p m/2}}F(f_p m \log p) > 0.18797.$$
We can numerical calculate the integral and find that
$$\int_0^\infty \frac{1-F(x)}{2\sinh{\frac{x}{2}}} \, dx < 0.70010.$$
Since the root discriminant of $KL$ is $\sqrt{653}$, we have:
\begin{align*}
B &= \gamma + \log{8\pi}  - \log \operatorname{rd}(KL) - \int_0^\infty \frac{1-F(x)}{2\sinh{\frac{x}{2}}} \, dx + 2 \sum_{p \in S} \sum_{m=1}^\infty \frac{\log p}{p^{f_p m/2}}F(f_p m \log p)\\
&> 0.57721 + 3.22417 - 3.24079 - 0.70010 + 0.18797 = 0.04846.
\end{align*}
Therefore, by Theorem \ref{UpperBound}, we get an upper bound for the class number of $KL$:
$$h_{KL} < \frac{2c\sqrt{\pi}}{nB} < \frac{2 \times 24.5 \sqrt{\pi}}{120 \times 0.04846} < 14.94.$$
Since the class number is an integer, we deteremine that
$$h_{KL} \leq 14.$$
Applying Corollary \ref{Cor}, we conclude that the class number of $KL$ is 1.
\end{proof}

\section{Summary of results and concluding remarks}
We have shown conditionally that the class number of $KL$ is one. Thus we know that $KL$ does not have any solvable unramified extensions. To prove unconditionally that the \'etale fundamental group of $K = \Q(\sqrt{653})$ is $A_5$, the remaining task is to show that $KL$ does not have a nonabelian simple unramified extension.

How might we rule out such nonabelian simple unramified extensions?  There are at least two possible approaches:  One is to gain sufficient knowledge of the prime ideals of the unramified extension in order to bound the degree of the extension, similar to our approach for bounding the degree of Hilbert class field.  But this is difficult:  We do not have access to a class field theory for nonabelian extensions as we have in the abelian case.  A possible alternative approach may be to use the theory of Galois representations to attempt to rule out possible nonabelian unramified extensions of $KL$.

\section{Acknowledgments}
The authors would like to thank Takeshi Tsuji and Hiroki Takahashi, the organizers of the 2014 RIMS Workshop on Algebraic Number Theory where the authors had the opportunity be introduced.

\medskip

\end{document}